\documentclass{amsart}
\usepackage{mathrsfs}
\usepackage{amssymb}
\usepackage{amsfonts}

\newtheorem{thm}{Theorem}[section]
\newtheorem{lemma}[thm]{Lemma}
\newtheorem{prop}[thm]{Proposition}
\newtheorem{cor}[thm]{Corollary}

\newtheorem{conjecture}[thm]{Conjecture}

\newtheorem{remark}[thm]{Remark}

\numberwithin{equation}{section}

\begin{document}

\title{A note on static spaces and related problems}

\author{Jie Qing}
\address{(Jie Qing) Department of Mathematics, University of California, Santa Cruz, CA 95064}
\email{qing@ucsc.edu}

\author{Wei Yuan}
\address{(Wei Yuan) Department of Mathematics, University of California, Santa Cruz, CA 95064}
\email{wyuan2@ucsc.edu}





\begin{abstract}
In this paper we study static spaces introduced in \cite{H-E, K-O, F-M, Kobayashi, Corvino} and Riemannian manifolds possessing solutions to the critical point equation
introduced in \cite{Besse, Huang, Chang_Huang, Yun_Huang}. In both cases on the manifolds there is a function $f$ satisfying the equation
$$
fRic = \nabla^2 f + \Phi g.
$$
With a similar idea used in \cite{C-C2, C-C1}, we have made progress in solving the classifying problem raised in \cite{F-M} of vacuum static spaces and in proving the conjecture made in \cite{Besse} about manifolds admitting solutions to the critical point equation in general dimensions. We obtain even stronger results in dimension 3. 
\end{abstract}

\maketitle



\section{Introduction}
 
Static space-times are the special and important global solutions to Einstein equations in general relativity. In this paper we are concerned with 
static space-times that carry a perfect fluid matter field as introduced in \cite{H-E} \cite{K-O}. One may include a cosmological constant 
to maintain mass-energy density to be nonnegative. A static space-time metric $\hat g = - f^2 dt^2 + g$ satisfies the Einstein equation
\begin{equation}\label{equ:einstein}
\hat{Ric} - \frac{1}{2}\hat{R} \hat{g} + \Lambda \hat g = - 8 \pi GT,
\end{equation}
for the energy-momentum-stress tensor $T = - \mu f^2 dt^2 - p g$ of a perfect fluid,
where $\mu$ and $p$ are nonnegative, time-indepdendent mass-energy density and pressure of the perfect fluid respectively. \\

A complete Riemannian manifold $(M^n,\ g)$ is said to be a static space (with perfect fluid) if there exists a
smooth function $f$ ($\not\equiv 0$) on $M^n$ such that $f$ solves the following static equation:
\begin{equation}\label{equ:static-1}
\nabla ^2 f - (Ric - \frac R{n-1}g)f - \frac{1}{n}( \frac R{n-1} f + \Delta f )g = 0.
\end{equation}
Particularly, $(M^n, \ g)$ is said to be a vacuum static space if \eqref{equ:static-1} reduces to
\begin{equation}\label{equ:vacuum-static-1}
\nabla^2 f -  (Ric - \frac R{n-1}g)f = 0.
\end{equation}

It is very interesting to notice that the vacuum static equation \eqref{equ:vacuum-static-1} are also considered by Fischer and Marsden \cite{F-M} in 
their study of the surjectivity of scalar curvature function from the space of Riemannian metrics (cf. \cite{Kobayashi, Shen, Corvino}).  \\

For static spaces, in \cite{K-O}, Kobayashi and Obata (cf. \cite{Lindblom} for $n=3$) showed that, nearby the hypersurface $f^{-1}(c)$ for a regular value $c$, 
a static metric $g$ is isometric to a warped metric of a constant curvature metric,  provided that $g$ is locally conformally flat. 
In \cite{F-M}, Fischer and Marsden raised the possibility of identifying all 
compact vacuum static spaces. In fact one now knows in dimension 3, besides flat tori $T^3$ and round spheres
$S^3$, $S^1\times S^2$ is also a compact vacuum static space. Later, in \cite{ejiri}, other warped metrics on $S^1\times_r S^2$ were found to be vacuum static. 
The open conjecture is that
those, possibly moduli some finite group, are all the compact vacuum static spaces. Please refer to \cite{Kobayashi, Shen, Lafontaine_1, Lafontaine} for progresses made 
in solving the classifying problem raised in \cite{F-M}. In short the classifying problem is solved \cite{Kobayashi, Lafontaine} for locally conformally flat static spaces.
But an easy calculation shows that $S^1({\frac 1{\sqrt{n-2}}})\times E^{n-1}$ for Einstein manifolds $E^{n-1}$
with scalar curvature $(n-1)(n-2)$ are compact vacuum static spaces, which are not locally conformally flat and therefore not accounted in \cite{Kobayashi}, when $n > 3$.\\

The critical point equation is introduced for the Hilbert-Einstein action on the space of conformal classes represented by 
Riemannian metrics with unit volume and constant Ricci scalar curvature in \cite{Besse} in an attempt to more efficiently identify Einstein metrics in two steps. 
Formally the Euler-Lagrangian equation of Hilbert-Einstein action on the space of  Riemannian metrics with unit volume and constant Ricci 
scalar curvature is 
$$
Ric - \frac 1n Rg = \nabla^2 f - (Ric - \frac  1{n-1} Rg)f.
$$
It may look more apparent  that it is related to the static equations \eqref{equ:static-1} and \eqref{equ:vacuum-static-1} if we replace $f$ by $f-1$ and consider the equation
\begin{equation}\label{equ:cpe}
\nabla^2 f - (Ric - \frac 1{n-1}Rg)f - \frac 1{n(n-1)}Rg = 0.
\end{equation}

A complete Riemannian manifold $(M^n, \ g)$ ($n\geq 3$) of constant Ricci scalar curvature is said to be CPE
if it admits a smooth solution $f$ ($\not\equiv 0$) to the critical point equation \eqref{equ:cpe} (cf. \cite{Besse, Huang, Chang_Huang, Yun_Huang}). In \cite{Besse} 
it conjectured that  a CPE metric is always Einstein.

\begin{conjecture} \label{besse} A CPE metric is always Einstein.
\end{conjecture}

It is clear that $(M^n, \ g)$ is Einstein if it admits a trivial solution $f\equiv -1$. Other CPE metrics with constant function $f$ are Ricci flat metrics.
$g$ is isometric to a round sphere metric if it is a Einstein CPE metric with a non-constant function $f$. Hence Conjecture \ref{besse} really says that a 
CPE metric with a non-constant solution $f$ to \eqref{equ:cpe} is isometric to a round sphere metric. Lafontaine in \cite{Lafontaine} verified Conjecture \ref{besse} 
when assuming metrics are locally conformally flat. Recently Chang, Hwang, and Yun in \cite{Yun_Huang} verified Conjecture \ref{besse} for metrics of harmonic 
curvature. \\

Recently in \cite{C-C2, C-C1} the authors studied Bach flat gradient Ricci solitons. Based on the similar
idea from \cite{C-C2, C-C1} we are able to solve the classifying problem raised in \cite{F-M}  for Bach flat vacuum static spaces in general dimensions. It is worth to mention that
we will include in our list the vacuum static spaces $S^1({\frac 1{\sqrt{n-2}}})\times E^{n-1}$ that were not accounted in
the lists given in \cite{K-O, Kobayashi} when $n > 3$. In the mean time, we are also able to verify Conjecture \ref{besse} for Bach flat CPE metrics. \\

Particularly in dimension 3, we establish an intriguing integral identity
\begin{equation}\label{section-4}
\int_{M^3} f^p {\mathcal C} = - \frac p4 \int_{M^3} f^p |C|^2
\end{equation}
where ${\mathcal C} = C_{ijk,}^{\quad \ ijk}$ is the complete divergence and $C_{ijk}$ is the Cotton tensor, on a compact 3-manifold $(M^3, \ g)$ admitting non-constant solution $f$ to the equation
\begin{equation}\label{equ:intro-both}
fR_{ij} = f_{i,j} + \Phi g_{ij}
\end{equation}
for some function $\Phi$. 
Therefore we are able to obtain stronger results for both static metrics and CPE metrics in dimension 3. For vacuum static spaces, based on the solutions to the 
corresponding  ODE given in \cite{Kobayashi}, we are able to solve the classifying problem raised in \cite{F-M}.

\begin{thm}\label{thm:f-m} Suppose that $(M^3, \ g)$ is compact vacuum static space with no boundary with nonnegative complete divergence $\mathcal C$ of the Cotton tensor. Then it must be one of the following up to a finite quotient: 
\begin{itemize}
\item Flat 3-manifolds;

\item $S^3$;

\item $S^1 \times S^2$;

\item $S^1 \times_{r} S^2$ for $g= ds^2 + r^2(s)g_{S^2}$, where $r(s)$ is a periodic function given in Eample 4 in \cite{Kobayashi}.
\end{itemize}
\end{thm}

Regarding Conjecture \ref{besse}, based on \cite{Lafontaine,Yun_Huang}, we prove the following:

\begin{thm}\label{thm:Besse} Conjecture \ref{besse} holds for compact Riemannian 3-manifold with no boundary with 
nonnegative complete divergence $\mathcal C$ of the Cotton tensor.
\end{thm} 

The organization of this paper is as follows: In section 2 we introduce Cotton tensors and Bach tensors on Riemannian manifolds. More importantly we introduce an augmented Cotton
tensor and prove an integral identity to allow us to most efficiently use the equation \eqref{equ:intro-both}. In section 3 we use the vanishing of the augmented Cotton tensor to establish the local splitting property. Then we give a complete classification for Bach flat vacuum static and verify Conjecture \ref{besse} for Bach flat CPE manifolds in general dimensions. In section 4 
we focus on dimension 3 and establish \eqref{section-4} and prove Theorem \ref{thm:f-m} and Theorem \ref{thm:Besse}.

\section{Preliminaries}

In this section we will use Bach flatness to force the vanishing of the augmented Cotton tensor $D$ as the authors did for gradient Ricci solitons in \cite{C-C2, C-C1}. 
To introduce the Bach curvature tensor of a Riemannian manifold $(M^n, \ g)$, 
we recall the well known decomposition of Riemann curvature tensor.
\begin{equation}\label{rie-weyl-schouten}
R_{ijkl} = W_{ijkl} + \frac 1{n-2}(S_{ik}g_{jl} - S_{il}g_{jk} - S_{jk}g_{il} + S_{jl}g_{ik})
\end{equation}
where $R_{ijkl}$ is the Riemann curvature tensor, $W_{ijkl}$ is the Weyl curvature tensor,
$$
S_{ij} = R_{ij} - \frac 1{2(n-1)}Rg_{ij}
$$
is Schouten curvature tensor, $R_{ij} = R_{ijk}^{\quad j}$ is Ricci curvature tensor, and $R= R_i^{\ i}$ is the Ricci scalar curvature.
Then the Cotton tensor $C$ is given as:
\begin{equation}\label{cotton}
C_{ijk} = S_{jk,i} - S_{ik,j}.
\end{equation}
The following consequence of Bianchi identity is often useful:
\begin{equation}\label{weyl-cotton}
W_{ijkl,}^{\quad \ \ l} =  - \frac{n-3}{n-2} C_{ijk}
\end{equation}
when $n\geq 4$. We are now ready to introduce the Bach curvature tensor  on a Riemannian manifold $(M^n, \ g)$ as follows:  
\begin{equation}\label{bach}
B_{jk} = \frac{1}{n-3} W_{ijkl,}^{\quad \ li} + \frac 1{n-2}S^{il} W_{ijkl},
\end{equation}
when $n\geq 4$. Using \eqref{weyl-cotton} we may extend the definition of Bach tensor in dimensions including $3$ as follows:
\begin{equation}\label{bach-3}
B_{jk} =  \frac 1{n-2}(-C_{ijk,}^{\quad i} +  S^{il} W_{ijkl}).
\end{equation}
Finally, as in \cite{K-O} and \cite{C-C2, C-C1}, we define the
following augmented Cotton tensor, which will play an important role in the calculations in this paper.
\begin{equation}\label{DCW}
 D_{ijk} = f^2 C_{ijk} -  f W_{ijkl} f^l.
\end{equation}
It is easy to see that $D_{ijk}$ is anti-symmetric about the indices $i$ and
$j$. In fact the following is a key observation (1.11) in \cite{K-O}. In order to treat both static equations \eqref{equ:static-1} and critical point equation
\eqref{equ:cpe} in the same way we need to rewrite them in a unified way. We first rewrite the static equation \eqref{equ:static-1} as follows:
\begin{equation}\label{equ:static-2}
f S = \nabla^2 f - \frac 1n (\Delta f - \frac {n-2}{2(n-1)}Rf)g.
\end{equation}
We then rewrite the critical point equation \eqref{equ:cpe} as follows:
\begin{equation}\label{equ:cpe-1}
fS = \nabla^2 f + (\frac {Rf}{2(n-1)}  - \frac R{n(n-1)} )g.
\end{equation}
In summary we will write both \eqref{equ:static-2} and \eqref{equ:cpe-1} in following form
\begin{equation}\label{equ:both}
fS = \nabla^2 f + \Phi g
\end{equation}
for a function $\Phi$ (this $\Phi$ is different from that in \eqref{equ:intro-both}).  

\begin{prop}\label{D_ijk} Suppose that $(M^n, \ g)$ is a Riemannain manifold admitting a smooth solution $f$ to the equation \eqref{equ:both}. Then
\begin{align}\label{equ:D}
D_{ijk} = \frac{1}{n-2} \underset{i,j}{Alt}\{(n-1)f_{i,k} f_j + \Psi_j g_{ik}\},
\end{align}
where $\underset{i,j}{Alt}$ means anti-symmetrizing with the indices $i$ and $j$, and
\begin{equation}\label{B-i}
\Psi_j = - (n-2) f\Phi_j + f_{j,l}f^l + n \Phi f_j.
\end{equation}
\end{prop}

\begin{proof}
It is a straightforward calculation based on the equation \eqref{equ:both}
and the definition of $D_{ijk}$ (cf. \cite{K-O}). For the convenience of readers we include some
calculations here. First we calculate
\begin{align*}
f^2 C_{ijk} & = f^2(S_{jk,i} - S_{ik,j}) \\
& = f(f_{k,ji} - f_{k,ij}) - f(S_{jk}f_i-S_{ik}f_j) + f(\Phi_ig_{jk} - \Phi_jg_{ik}) \\
\end{align*}
Then recall the Ricci identity
$$
f_{k,ji} - f_{k,ij} = f_lR^l_{\ \ kji} = R_{ijkl}f^l
$$
and conclude that
$$
f_{k,ji} - f_{k, ij} =  W_{ijkl}f^l  + \frac 1{n-2}(S_{jl}g_{ik} - S_{il}g_{jk})f^l + \frac 1{n-2}(S_{ik}f_j - S_{jk}f_i).
$$
Hence we obtain 
$$
(n-2)(f^2C_{ijk} - fW_{ijkl}f^l) =  \underset{i,j}{Alt}\{(n-1)fS_{ik} f_j + g_{ik}\tilde\Psi_j\}
$$
for 
$$
\tilde\Psi_j = - (n-2) f\Phi_j + fS_{jl}f^l.
$$
From here, using the equation \eqref{equ:both}, we complete the proof of \eqref{equ:D}.
\end{proof}

\begin{remark}\label{rem:Phi} Note that 
\begin{equation}\label{static-Phi}
\Phi = \frac 1n(\frac {n-2}{2(n-1)}Rf - \Delta f) \text{ and } \ \Psi_j = f_{j,l}f^l - \Delta f f_j 
\end{equation}
for static metrics and
\begin{equation}\label{cpe-Phi}
\Phi = \frac {Rf}{2(n-1)} - \frac R{n(n-1)} \text{ and } \ \Psi_j = f_{j,l}f^l - \Delta f f_j 
\end{equation}
for CPE metrics. It is very intriguing to see that $\Psi$ is the same for the both cases.
\end{remark}
Then we can rewrite the Bach tensor as follows:

\begin{prop}\label{prop:Bach} Suppose that $(M^n, \ g)$ is a Riemannian manifold admitting a smooth solution $f$ to the equation \eqref{equ:both}. Then
\begin{equation}\label{D-bach}
(n-2)B_{jk} = -  \nabla^i (\frac{D_{ijk}}{f^2}) + \frac{n-3}{n-2} C_{lkj} \frac{f^l}{f} + W_{ijkl} \frac{f^i f^l}{f^2} .
\end{equation}
\end{prop}

\begin{proof}
It is straightforward to calculate that, from the definition \eqref{DCW},
\begin{align*}
(n-2) B_{jk} &= - C_{ijk,}^{\quad i} + S^{il} W_{ijkl} = -  \nabla^i (
W_{ijkl} \frac{f^l}{f} + \frac{1}{f^2} D_{ijk}) + S^{il} W_{ijkl}\\
&=  - \nabla^i (\frac{D_{ijk}}{f^2})  - W_{ijkl,}^{\quad \ \ i}
\frac{f^l}{f} + W_{ijkl}( S^{il} - \frac{f^{i,l}}{f} +\frac{f^i f^l}{f^2})\\
&=  - \nabla^i (\frac{D_{ijk}}{f^2}) + \frac{n-3}{n-2} C_{lkj}
\frac{f^l}{f} + W_{ijkl} \frac{f^if^l}{f^2}.
\end{align*}
\end{proof}

Now, as a consequence of \eqref{D-bach}, we can state one of  the key identities in this paper. To state that we introduce some notations.
We will denote the level set
$$
M_c = \{x\in M^n: f(x) = c\}
$$
and
$$
M_{c_1,c_2} = \{x\in M^n: c_1 <  f(x) < c_2\}.
$$

\begin{prop} \label{int_id_1} Suppose that $(M^n, \ g)$ is a Riemannian manifold admitting a smooth solution $f$ to the equation
\eqref{equ:both}. Let $c_1$ and $c_2$ be two regular values for the function $f$ 
and two level sets $M_{c_1}$ and $M_{c_2}$ be compact. Then, for all $p \geq 2$, we have the 
following integral identity:
\begin{align}\label{equ:main}
\int_{M_{c_1,c_2}} f^p B_{jk} f^{j,k} = 
\frac{1}{2(n-1)} \int_{M_{c_1,c_2}} f^{p-2} |D|^2.
\end{align}
\end{prop}

\begin{proof}
By the anti-symmetries of $W_{ijkl}$, $C_{ijk}$ and $D_{ijk}$, from \eqref{D-bach} one gets 
\begin{equation*}
B_{jk} f^{j,k} = - \frac{1}{n-2}D_{ijk,}^{\quad i} \frac{f^j f^k}{f^2}.
\end{equation*}
Applying integrating by parts, we get
\begin{equation*}
(n-2)\int_{M_{c_1,c_2}} f^p B_{jk} f^{j,k} 
=  \int_{M_{c_1,c_2}} D_{ijk} \nabla^i(f^{p-2} f^j f^k ).
\end{equation*}
Again, due to the anti-symmetries and trace-free properties of Cotton tensor $C$ and the augmented Cotton tensor $D$, we arrive at
\eqref{equ:main}
\begin{align*}
(n-2)\int_{M_{c_1,c_2}} f^p B_{jk}   & f^{j,k} 
=  \int_{M_{c_1,c_2}} f^{p-2}D_{ijk}f^{i,k} f^j\\
=& \frac{n-2}{2(n-1)} \int_{M_{c_1,c_2}} f^{p-2} |D|^2.
\end{align*}
\end{proof}

Consequently we obtain the following important initial step to understand the geometric
structure of a Riemannain manifold admitting a smooth solution to the equation \eqref{equ:both}.

\begin{cor}\label{D-cotton}
\label{Bach_flat} The augmented Cotton tensor $D$ vanishes identically on a
Bach flat manifold admitting a smooth non-constant solution $f$ to the equation
\eqref{equ:both}, provided that each level set $f^{-1}(c)$ is
compact for any regular value $c$.
\end{cor}

\section{Bach flat cases}

In this section, based on Corollary \ref{D-cotton}, we investigate geometric structure of a
Bach flat manifold admitting a smooth non-constant solution $f$ to the equation
\eqref{equ:static-1} or \eqref{equ:cpe}. To facilitate our local calculations we need to choose local
frames and set notations. \\

For a regular value $c$, we denote the level set $f^{-1}(c)$ as $\Sigma$,
$W:=|\nabla f|^2$,  and $e_n := \frac{\nabla f}{|\nabla f|}$ as the unit normal to $\Sigma$.
We then choose an orthonormal frame 
$$
\{e_1,e_2,\cdots,e_{n-1},e_n\}
$$ 
along $\Sigma$. We will use Greek letters to denote the index from $1$ to $n-1$, while Latin letters for the index
from $1$ to $n$.
Then the second fundamental form of $\Sigma$ is
\begin{align}
h_{\alpha\beta} = \langle \nabla_{e_\alpha} e_\beta, e_n \rangle = - \langle e_\beta,
\nabla_{e_\alpha} e_n \rangle = -\frac{f_{\alpha, \beta}}{|\nabla f|},
\end{align}
the mean curvature is
\begin{align}\label{meancurvature}
H = g^{\alpha\beta}h_{\alpha\beta} =
W^{-\frac{1}{2}}( f_{n,n} - \Delta f),
\end{align}
and  the square of the norm of the second fundamental form is
\begin{align}
|A|^2 = h_{\alpha\beta}h^{\alpha\beta } = W^{-1} \sum_{\alpha,\beta =1}^{n-1}|f_{\alpha, \beta}|^2.
\end{align}
Furthermore
\begin{align}
|\nabla^\Sigma W|^2 = 4W \sum_{\alpha=1}^{n-1}|f_{n,\alpha}|^2
\end{align}
and
\begin{align}
 |\nabla_n W|^2 = 4W|f_{n,n}|^2.
\end{align}
Now we are ready to prove another key identity in this paper.

\begin{prop}\label{2-key} Suppose that $(M^n, \ g)$ is a Riemannian manifold admitting a non-constant solution
to either \eqref{equ:static-1} or \eqref{equ:cpe}. Then the following identity holds:
\begin{align}\label{levelset}
|D|^2 = 2\frac{(n-1)^2}{(n-2)^2}W^2 |A - \frac{H}{n-1}
g^\Sigma|^2 + \frac{n-1}{2(n-2)}|\nabla^\Sigma W|^2.
\end{align}
\end{prop}

\begin{proof}
By Proposition \ref{D_ijk}, we have
\begin{align*}
(n-2)^2 |D|^2  & = 2(n-1)^2 |\nabla f|^2|\nabla^2 f|^2 + 2(n-1)|\Psi|^2 - 2(n-1)^2f_{k,i}f^if^{k,j}f_j\\
& \quad\quad  + 4(n-1)(\Delta f \nabla f\cdot \Psi - f_{i,j}f^i\Psi^j) \\
& = 2(n-1)^2|\nabla f|^4|A - \frac H{n-1}g^\Sigma|^2 + 2(n-1)|\nabla f|^4 H^2 \\
& \quad\quad + 2(n-1)^2 |\nabla f|^2\sum_{\alpha =1}^{n-1}|f_{n, \alpha}|^2 \\
& \quad\quad  + 2(n-1)|\Psi|^2 + 4(n-1)(\Delta f \nabla f\cdot \Psi - f_{i,j}f^i\Psi^j) 
 \end{align*}
Because
\begin{align*}
         |\nabla^2 f|^2 = \sum_{\alpha,\beta =1}^{n-1}|f_{\alpha, \beta}|^2 +
         2 \sum_{\alpha=1}^{n-1}|f_{n, \alpha}|^2 + |f_{n,n}|^2.
\end{align*}
We also calculate, due to Remark \ref{rem:Phi}, 
\begin{align*}
|\Psi|^2 = |\nabla f|^2(\sum_{\alpha = 1}^{n-1}|f_{n, \alpha}|^2 + |f_{n,n} - \Delta f|^2)
\end{align*}
and
\begin{align*}
\Delta f \nabla f\cdot \Psi- f_{i,j}f^i\Psi^j   & = - |\nabla f|^2 (f_{n,n} - \Delta f )(f_{n,n}  - \Delta f)\\
& - |\nabla f|^2 \sum_{\alpha = 1}^{n-1} |f_{n, \alpha}|^2
\end{align*}
Therefore
\begin{align*}
\frac{(n-2)^2}{n-1} |D|^2 = 2(n-1)W^2|A - \frac H{n-1}g^\Sigma|^2 + \frac {n-2}2 |\nabla^\Sigma W|^2    
\end{align*}

\end{proof}
An immediate consequence is following:

\begin{cor}\label{mean-curvature} Suppose that $(M^n, \ g)$ ($n\geq 3$) is a Riemannian manifold admitting a non-constant solution
to either \eqref{equ:static-1} or \eqref{equ:cpe}.
And suppose that  the augmented Cotton tensor $D$ vanishes. Then the level set $\Sigma$ is umbilical and the
mean curvature $H$ is constant.
\end{cor}

\begin{proof}By the assumption we know that the solution $f$ can not be a constant. Therefore it follows from Lemma \ref{2-key}
that the level set $\Sigma$ is umbilical and $W$ is a constant along $\Sigma$ in the light of \eqref{levelset}. In fact
\begin{align*}
      |\nabla^\Sigma W|^2 = \sum_{\alpha= 1}^{n-1} |\nabla_\alpha W|^2
              = 4 W \sum_{\alpha = 1}^{n-1} |f_{n, \alpha}|^2.
\end{align*}
Hence, according to the equation \eqref{equ:both}, we conclude that
$R_{\alpha n} = 0$, for $\alpha=1,2,\cdots, n-1$.  On the other hand, by contracting the Codazzi equations we get
\begin{align*}
0= R_{\alpha n} = \frac{n-2}{n-1} \nabla^\Sigma_\alpha H, \ \ \ \alpha=1,2,\dots,n-1.
\end{align*}
Therefore the mean curvature $H$ is constant along $\Sigma$.
\end{proof}

Next we show the constancy of $R$ and $\Delta f$ along $\Sigma$.

\begin{lemma}\label{dR} Suppose that $(M^n, \ g)$ ($n\geq 3$) is a static space or a CPE metric with a non-constant
function $f$. Then
\begin{align}
\nabla^\Sigma R = \nabla^\Sigma \Delta f = 0.
\end{align}
\end{lemma}
\begin{proof} The statement of this lemma is obviously true for a CPE metric. For a static metric, taking divergence of the static equation 
\eqref{equ:static-1}, we have
\begin{align}
d(Rf + (n-1) \Delta f) = \frac{n}{2} f dR,
\end{align}
which implies
\begin{equation}\label{equ:dR}
(\frac{n}{2} - 1) f dR = R df + (n-1) d\Delta f.
\end{equation}
Taking exterior differential of the two sides of the above equation, we get $df \wedge dR = 0.$
Hence, by Cartan's lemma, there exists a smooth function $\phi$ such
that $dR = \phi df$, which implies $\nabla^\Sigma_\alpha R = \nabla_\alpha R = \phi \nabla_\alpha f
=0$, i.e. $\nabla^\Sigma R = 0$. Consequently, in the light of \eqref{equ:dR}, one also gets
$\nabla^\Sigma \Delta f = 0$.
\end{proof}

Consequently we know that the level set $\Sigma$ is of constant scalar curvature if the augmented Cotton tensor
vanishes. 

\begin{cor}\label{3dim} Suppose that $(M^n, \ g)$ ($n\geq 3$) is a static space or a CPE metric with a non-constant
function $f$.  And suppose that the augmented Cotton tensor 
$D$ vanishes identically.  Then the level set $\Sigma$ is of constant scalar curvature.
\end{cor}

\begin{proof}  Recall Gauss equation
\begin{align*}
      R^\Sigma = R - 2 R_{nn} + H^2 - |A|^2.
\end{align*}
Hence it suffices to show that $R_{nn}$ to be constant along $\Sigma$ in the light of Corollary \ref{mean-curvature} and Lemma \ref{dR}.
To do that we first realize that $f_{n,n}$ is constant from \eqref{meancurvature}. Then the conclusion follows from the static equation \eqref{equ:static-1}
or critical point equation \eqref{equ:cpe}.
\end{proof}

To work a bit harder we can show that in fact the level set $\Sigma$ is Einstein when the augmented Cotton tensor vanishes.

\begin{prop}\label{>3} Suppose that $(M^n, \ g)$ ($n\geq 3$) is a static space or a CPE metric with a non-constant
function $f$.  And suppose that the augmented Cotton tensor $D$ vanishes identically. Then the level set $\Sigma$ is Einstein.
\end{prop}

\begin{proof} We start with the assumption that $D=0$. Hence, from the definition \eqref{DCW}, we have
$$
W_{ijkl}f^l f^i= f C_{ijk}f^i.
$$
On the other hand, from Bach flatness and Proposition \ref{prop:Bach}, we also have
$$
 W_{ijkl}f^if^l =  - \frac {n-3}{n-2} f C_{ijk}f^i.
 $$
Therefore we can conclude that $W_{ijkl}f^if^l = 0$, that is, $W_{njkn} = 0$. Using the Riemann curvature decomposition we derive
\begin{align*}
R_{\alpha n\beta n} & = W_{\alpha n\beta n} + \frac 1{n-2}R_{\alpha\beta}  + \frac 1{n-2}(S_{nn} - \frac 1{2(n-1)} R) g_{\alpha\beta}\\
& = \frac 1{n-2}R_{\alpha\beta}  + \frac 1{n-2}(R_{nn} - \frac 1{n-1} R) g_{\alpha\beta}.
\end{align*}
Meanwhile, from the equation \eqref{equ:both}, we obtain
\begin{align*}
R_{\alpha\beta} &= \frac{\nabla_\alpha\nabla_\beta f}{f} + \frac{1}{n} ( R -
\frac{\Delta f}{f} ) g_{\alpha\beta} \\
&= -\frac{|\nabla f|}{f} h_{\alpha\beta} + \frac{1}{n} ( R - \frac{\Delta
f}{f} ) g_{\alpha\beta}\\ &= (\frac{1}{n} ( R - \frac{\Delta f}{f} )-
\frac{H}{n-1}\frac{|\nabla f|}{f}) g_{\alpha\beta}.
\end{align*}
Finally, using Gauss equation,
\begin{align*}
R^\Sigma_{\alpha\beta} = R_{\alpha\beta} - R_{\alpha n\beta n} + H h_{\alpha\beta} - h_{\alpha\gamma} h^\gamma_{\ \ \beta}
\end{align*}
we can conclude that $\Sigma$ is Einstein by Schur's lemma when $n\geq 4$. Notice that Corollary \ref{3dim} implies the proposition 
when $n=3$. Thus the proof is complete.
\end{proof}

We now summarize what we have achieved in the following local splitting result for the
geometric structure of a static metric or a CPE metric (cf. Theorem 3.1 in \cite{K-O}).

\begin{thm}\label{thm:>3}
Suppose that $(M^n, \ g)$ is a static space or CPE manifold with non-constant function $f$ and compact level set $f^{-1}(c)$
for a given regular value $c$. And assume it is
Bach flat. Then
\begin{align*}
g = ds^2 + (r(s))^2 g_E,
\end{align*}
nearby the level set $f^{-1}(c)$, where $ds = \frac{df}{|d f|}$, $(r(s))^2 g_E =
g|_{f^{-1}(c)}$ and $g_E$ is an Einstein metric.
\end{thm}

Consequently, based on the solutions to the corresponding ODE given in \cite{Kobayashi}, one gets the classification theorem for 
Bach flat vacuum static spaces. Notice that the function $f$ and the warping factor $r$ still satisfy the same ODE system:
\begin{equation*}
\begin{cases}
f'' + (n-1)\frac {r'}r f' + \frac R{n-1}f & = 0\\
r'f'- r'' f & = 0
\end{cases}
\end{equation*}
which is (1.9) in \cite{Kobayashi}. It is remarkable that Kobayashi was able to find the integrals and completely solved it. The solutions depend on
the constants $R$, 
$$
a = r^{n-1}r'' + \frac R{n(n-1)}r^n,
$$
and
$$
k = (r')^2 + \frac R{n(n-1)}r^2 + \frac {2a}{n-2}r^{2-n}.
$$
The horizontal slice $E$ is Einstein with $Ric = (n-2)kg_E$ here.

\begin{thm}\label{classification}
Let $(M^n,g,f)$ be a Bach flat vanuum static space with compact
level sets ($n \geq 3$). Then up to a finite quotient
and appropriate scaling,

(i) $f$ is a non-zero constant if and only if $M$ is Ricci flat;

(ii) $f$ is non-constant if and only if $M$ is isometric to

     \begin{itemize}
        \item $S^n$;
        \item $\mathbb{H}^n$;
        \item the warped product cases.
     \end{itemize}
In the warped product cases, we can divide again into compact and non-compact ones. For the compact ones
$S^1 \times_r E$ with metric $g = ds^2 + (r(s))^2 g_E,$  $r(s)$ appears to be one of the following:

\begin{itemize}
\item $r(s)$ is a constant and $E$ is an arbitrary compact Einstein
manifold of positive scalar curvature without boundary (cf. Example 2 in \cite{Kobayashi});
\item $r(s)$ is non-constant and periodic and $E$ is an arbitrary compact
Einstein manifold of positive scalar curvature without boundary (cf. Example 4 in \cite{Kobayashi}).
\end{itemize}
For the non-compact ones $\mathbb{R} \times_r E$ with metric
$g = ds^2 + (r(s))^2 g_E,$ $r(s)$ appears to be one of the following:

\begin{itemize}
\item $r(s)$ is a constant and $E$ is an arbitrary compact Einstein
manifold without boundary (cf. Example 1 in \cite{Kobayashi});
\item $r(s)$ is non-constant and peroidic and $E$ is an arbitrary compact
Einstein manifold of positive scalar curvature without boundary (cf. Example 3 in \cite{Kobayashi});
\item $r(s)$ is given in Proposition 2.5 in \cite{Kobayashi} and $E$ is an arbitrary compact Einstein
manifold without boundary (cf. Example 5 in \cite{Kobayashi}) .
\end{itemize}
\end{thm}

\begin{remark} We would like to mention again, since we only assume Bach flatness, our
list includes the warped metric where the level sets are only Einstein instead of constant curvature
as in \cite{Kobayashi, Lafontaine} .
\end{remark}

On the other hand, as a consequence of Theorem \ref{thm:>3}, a Bach flat CPE metric
turns out to be of harmonic Riemann curvature. Namely,

\begin{lemma}\label{warped_cotton}
Suppose the metric $g$ is a CPE metric satisfying assumptions in Theorem \ref{thm:>3}.  
Then the Cotton tensor $C$ of $g$ vanishes identically and therefore $g$ is of harmonic Riemann curvature.
\end{lemma}

\begin{proof}
We simply choose a local coordinate system $\{ \partial_1, \partial_2, \cdots, \partial_{n-1},\partial_n = \partial_s \}$ and calculate directly. It is easily seen that
\begin{align*}
C_{\alpha\beta\gamma}=C_{\alpha\beta n} = C_{n\beta n} = 0.
\end{align*}
The only term that needs some effort is $C_{n\beta\gamma}$, which in fact is seen to be zero from \eqref{D-bach} and the 
fact that both Bach tensor and the augmented Cotton tensor $D$ are identically zero. Notice that $W_{njkn}$ is known to be identically zero from 
the proof of Proposition \ref{>3}. To see the harmonicity of Riemann curvature we calculate as follows:
\begin{align*}
R_{ijkl,}^{\quad \ \ l} & = W_{ijkl,}^{\quad \ \ l} + \frac 1{n-2}(S_{ik}g_{jl} - S_{il}g_{jk} - S_{jk}g_{il} + S_{jl}g_{ik})_{,}^{ \ l}\\
& = - \frac {n-3}{n-2}C_{ijk} + \frac 1{n-2} (S_{ik,j} - S_{jk,i}) \\
& = -C_{ijk} = 0
\end{align*}
using the fact that the Ricci scalar curvature $R$ is constant.
\end{proof}

Then, using the result in \cite{Yun_Huang}, we can verify Conjecture \ref{besse} for Bach flat manifolds.

\begin{thm}\label{besse-conj-1} Suppose that $(M^n, \ g)$ ($n\geq 3$) is Bach flat CPE manifold admitting a non-constant solution to \eqref{equ:cpe}. 
Then $(M^n, \ g)$ is isometric to a round sphere.
\end{thm}

\section{In dimension 3}

In dimension 3 we recall that the Bach tensor is given as the divergence of the Cotton tensor in \eqref{bach-3}. What we will do in this section is to establish 
another integral identity on compact manifold with a static metric or a CPE metric. Then we will be able to conclude that the full divergence
$B_{ij,}^{\quad ij}$ of the Bach tensor (the full divergence $C_{ijk,}^{\quad \ ijk}$  of the Cotton tensor) vanishes if and only if the Cotton tensor vanishes in dimension
3 for a static metric as well as a CPE metric on a compact manifold.  

\begin{prop} Suppose that $(M^n, \ g)$ ($n\geq 3$) is a compact Riemannian manifold with no boundary admitting a non-constant smooth 
solution to \eqref{equ:both}. Then,  for any $p \geq 2$, we have the following integral identity:
\begin{equation} \label{int_id_2} 
\int_M f^p B_{ij,}^{\quad ij} = -
\frac{p(n-4)}{2(n-1)(n-2)} \int_M  f^{p-2} D\cdot C .
\end{equation}
\end{prop}

\begin{proof} First, applying integrating by part twice, we get
\begin{equation}\label{1-equ}
\int_M f^p  B_{ij} f^i f^j = \frac{1}{(p+1)(p+2)}\int_M f^{p+2}B_{ij,}^{\quad ij}
 - \frac{1}{p+1}\int_M f^{p+1} B_{ij}f^{i,j}.
\end{equation}
Then we use Proposition \ref{prop:Bach} to calculate the second term in the right hand side of the above equation.
Namely,
\begin{align*}
(n-2) \int_M f^{p+1} B_{ij}f^{i,j} &  = - \int_M f^{p+1}  \nabla^k (\frac{D_{kij}}{f^2})
f^{i,j} + \frac {n-3}{n-2} \int_M f^p C_{kij}
f^kf^{i,j} \\ & +  \int_M f^{p-1} W_{iklj} f^k f^l f^{i,j}.
\end{align*}
Now we deal with each term separately. For the first term, we perform once again integrating by part and get:
\begin{align*}
\int_M f^{p+1} \nabla^k (\frac{D_{kij}}{f^2})
f^{i,j} = \frac{(n-2)(p+2)}{2(n-1)} \int_M f^{p-2}
|D|^2 -  \frac 12 \int_M f^p D\cdot C.
\end{align*}
For the second term we simply use Proposition \ref{D_ijk}:
\begin{align*}
\int_M f^p C_{kij} f^{i,j}f^k =  - \frac {n-2} {2(n-1)} \int_M f^{p} D\cdot C.
\end{align*}
And for the last term, we use the definition of Bach tensor and again perform more integrating by part:
\begin{align*}
\int_M f^{p-1} W_{kijl} f^kf^lf^{i,j}  =  (n-2) \int_M f^p B_{ij}f^if^j
- \frac{n-2}{2(n-1)} \int_M f^{p} D\cdot C.
\end{align*}
Combining all the three terms together, we get
\begin{align}\label{2-equ}
\int_M f^{p+1} B_{ij}f^{i,j} = -\frac {n-4}{2(n-1)(n-2)} \int_M f^p D\cdot C -(p+1) \int_M f^pB_{ij}f^if^j,
\end{align}
where we have applied Proposition \ref{int_id_1}. Going back and rewriting \eqref{1-equ} as follows:
\begin{align*}
\frac{1}{p+2} \int_M f^{p+2} B_{ij,}^{\quad ij} 
= (p+1)\int_M f^p  B_{ij}f^if^j + \int_M f^{p+1} B_{ij}f^{i,j},
\end{align*}
which implies, from \eqref{2-equ},
$$
\frac 1{p+2}\int_M f^{p+2}B_{ij,}^{\quad ij} = - \frac {n-4}{2(n-1)(n-2)}\int_M f^p D\cdot C.
$$
So the proof is complete.
\end{proof}

In particular, when $n=3$, we obtain

\begin{cor} Suppose that $(M^3, \ g)$ is a compact Riemannian manifold with no boundary admitting
a non-constant smooth solution to \eqref{equ:both}. Then, for any $p\geq 2$,
\begin{equation}\label{equ:c=d}
\int_M f^p C_{ijk,}^{\quad \ ijk} =  - \frac p4 \int_M f^p |C|^2.
\end{equation}
\end{cor}

Hence we have improved Theorem \ref{thm:>3} in dimension 3.

\begin{thm} Suppose that $(M^3, \ g)$ is a compact Riemannian manifold with no boundary with a static metric 
or CPE metric and non-constant function $f$. If $C_{ijk,}^{\quad \ ijk}$ vanishes identically, then the Cotton
tensor vanishes identically and therefore Theorem \ref{thm:>3} holds.
\end{thm}

More interestingly we have the improved version of Theorem \ref{classification}, which gives a
partial answer to the Fischer-Marsden's problem (cf. \cite{F-M}).

\begin{thm} Suppose that $(M^3, \ g)$ is a compact vacuum static space with
$C_{ijk,}^{\quad \ ijk}$ vanishing identically.  Then the vacuum static space must
be one of the following up to a finite quotient and appropriate
scaling,

(i) Flat space;

(ii) $S^n$;

(iii) $S^1 \times S^2$;

(iv) $S^1 \times_{r} S^2$ with warped metric $g= ds^2 + r^2(s)g_{S^2}$, where $r(s)$ is a periodic function 
given in Example 4 in \cite{Kobayashi}.
\end{thm}

Similarly we have the improved version of Theorem \ref{besse-conj-1} as follows:

\begin{thm}  Conjecture \ref{besse} holds for compact 3-manifold $(M^3, \ g)$ with no boundary satisfying 
$C_{ijk,}^{\quad \ ijk} = 0$.
\end{thm}

\paragraph{\textbf{Acknowledgement}}
The second author would like to thank Dr. David DeConde for his helpful discussion and valuable suggestions.

\bibliographystyle{amsplain}

\end{document}